\date{}
\documentclass[11pt]{elsarticle}
\usepackage{amsfonts,amsmath,amssymb,amsthm, color, enumerate}
\def\fullpage {
\addtolength{\topmargin}{-2 cm}
\addtolength{\oddsidemargin}{-1.6 cm} \addtolength{\textwidth}{+3.3 cm}
\addtolength{\textheight}{+3.5 cm}}
\fullpage
\journal{European Journal of Combinatorics}

\newcommand{\Z}{\mathbb{Z}}

\newtheorem{theorem}{Theorem}

\newtheorem{lemma}{Lemma}

\newtheorem{proposition}{Proposition}
\begin{document}

\begin{frontmatter}

\title{On sets of integers with restrictions on their products}

\author{Michael Tait\corref{cor1}}\ead{mtait@math.ucsd.edu}

\author{Jacques Verstra\"{e}te\corref{cor2}}\ead{jverstra@math.ucsd.edu}

\cortext[cor1]{Corresponding author}
\cortext[cor2]{Research supported by NSF Grant DMS-1101489.}

\address{Department of Mathematics\\
University of California at San Diego\\
9500 Gilman Drive, La Jolla, California 92093-0112, USA}

\begin{abstract}
A {\em product-injective labeling} of a graph $G$ is an injection $\chi: V(G) \to \Z$ such that $\chi(u)\chi(v) \not= \chi(x)\chi(y)$ for any distinct edges $uv, xy\in E(G)$. Let $P(G)$ be the smallest $N \geq 1$ such that there exists a product-injective labeling $\chi : V(G) \rightarrow [N]$. Let $P(n,d)$ be the maximum possible value of $P(G)$ over $n$-vertex graphs $G$ of maximum degree at most $d$. 
In this paper, we determine the asymptotic value of $P(n,d)$ for all but a small range of values of $d$ relative to $n$. Specifically, we show that there exist constants $a,b > 0$ such that 
$P(n,d) \sim n$ if $d \leq \sqrt{n}(\log n)^{-a}$ and $P(n,d) \sim n\log n$ if $d \geq \sqrt{n}(\log n)^{b}$.
\end{abstract}

\end{frontmatter}

\section{Introduction}

Let $G$ be a graph. A {\em product-injective labeling} of $G$ is an injection $\chi: V(G) \to \Z$ such that $\chi(u)\cdot \chi(v) \not= \chi(x)\cdot \chi(y)$ for distinct edges $uv,xy \in E(G)$.
Let $P(G)$ denote the smallest positive integer $N$ such that there is a product-injective labeling $\chi : V(G) \rightarrow [N]$. 
In this paper, our main results give asymptotically tight bounds on $P(G)$ relative to the maximum degree $d$
and number of vertices of the graph $G$, for all but a small range of values of $d \leq n - 1$. Let $P(n,d)$ be the maximum possible value of $P(G)$ over $n$-vertex graphs $G$ of maximum degree at most $d$. 
Specifically, we prove the following theorem:

\begin{theorem}\label{main} 
There exist constants $a,b > 0$ such that {\rm (i)} $P(n,d) \sim n$ if $d \leq \sqrt{n}(\log n)^{-a}$ and {\rm (ii)} $P(n,d) \sim n\log n$ if $d \geq \sqrt{n}(\log n)^{b}$.
\end{theorem}

An old result of Erd\H{o}s \cite{E} implies $P(K_n) \sim n\log n$, whereas Theorem \ref{main} shows that $P(G) \sim n\log n$ for graphs which are much sparser than $K_n$. 
The labeling of the vertices of any $n$-vertex graph $G$ with the first $n$ prime numbers is always a product-injective labeling from $[N]$ where $N \sim n\log n$, 
via the Prime Number Theorem. An analogous result to Theorem \ref{main} for labelings of graphs where differences or sums are required to be distinct for distinct edges was obtained by Bollob\'{a}s and Pikhurko~\cite{BP}, where a change in behavior was also observed around $d = \sqrt{n}$. Theorem \ref{main} will be proved with $a > \log 2$ and $b > 4.5$; while our method allows these values to be slightly improved, new ideas would be needed to determine $P(n,d)$ for all the intermediate values of $d$. In fact, the proof of Theorem \ref{main}(ii) establishes the much stronger result that if $G$ is the random graph on $n$ vertices with edge-probability $d/n$ and $d \geq \sqrt{n}(\log n)^b$, then $P(G) \sim n\log n$ almost surely as $n \rightarrow \infty$. We also remark that Theorem \ref{main} determines the maximum value of $P(G)$ over $n$-vertex graphs with $m$ edges for 
almost all possible values of $m$.

\medskip

{\bf Notation and Organization.} The paper is organized as follows. The proof of Theorem \ref{main} uses the modified local lemma, which we state in Section \ref{lemmas}, together
with some facts on the distribution of the number of divisors function $\tau$. The proof of Theorem \ref{main}(i) is given in Section \ref{thm1}, and
Theorem \ref{main}(ii) is proved in Section \ref{thm2}.

\medskip

For real numbers $y \geq x \geq 1$, we use the notation $[x] = \{1,2,\dots,\lfloor x\rfloor\}$ and $[x,y] = [y] \backslash [x]$. Let the Erd\H{o}s-R\'{e}nyi random graph, $G_{n,p}$, be a graph chosen from the probability space $\mathcal{G}_{n,p}$  where edges of $K_n$ are present in $G_{n,p}$ independently with probability $p$. For background on random graphs, see Bollob\'{a}s~\cite{B}.  All logarithms will be in the natural base, and all graphs will be simple and finite. If $(A_n)_{n \in \mathbb N}$ is a sequence of events in a probability space, then we say $A_n$ occurs almost surely as $n \rightarrow \infty$ if $\lim_{n \rightarrow \infty} P(A_n) = 1$.
We write $f(n) \sim g(n)$ for functions $f,g : \mathbb Z^+ \rightarrow \mathbb R$ if $f(n)/g(n) \rightarrow 1$ as $n \rightarrow \infty$ and 
$f(n) \ll g(n)$ if $f(n)/g(n) \rightarrow 0$ as $n \rightarrow \infty$.

\section{Preliminaries} \label{lemmas}

To prove Theorem \ref{main}, we make use of a probabilistic result known as the modified local lemma, together
with some well-known facts from analytic number theory regarding the number of divisors of positive integers.

\medskip

{\bf Modified local lemma.} The modified local lemma, which is a version of the Lov\'{a}sz Local Lemma (see Alon and Spencer~\cite{AS}, page 65), is used in
the following form:

\begin{proposition}\label{mod}
Let $A_1,...,A_n$ be events in a probability space and for each $i \in [n]$, let $(J_i,K_i)$ be a partition of $[n] \backslash \{i\}$.
If there exists $\gamma \in [0,1)$ such that
\[
\mathbb{P}\left(A_i \; | \; \bigcap_{k\in K} \overline{A}_k\right) \leq \gamma (1 - \mathrm{max}|J_i| \gamma)\]
then
\[ \mathbb{P}\left( \bigcap_{i = 1}^n \overline{A}_i\right) > 0.\]
\end{proposition}

\medskip

{\bf Distribution of the number of divisors.} For a natural number $k$, let $\tau(k)$ be the number of divisors of $k$, and let $\Omega(k)$ be the number of prime power divisors of $k$. The Hardy-Ramanujan Theorem \cite{HR} gives
$|\{x \leq N : |\Omega(x) - \log\log N| \gg \sqrt{\log \log N}\}| \ll N$ as $N \rightarrow \infty$.
Since $\tau(n) \leq 2^{\Omega(n)}$, one has the following result:

\begin{proposition}\label{tau} Let $\omega(N) \rightarrow \infty$ as $N \rightarrow \infty$. Then
\[ |\{n \leq N : \tau(n) \geq (\log N)^{\log 2}2^{\omega(N)\sqrt{\log \log N}}\}| \ll N.\]
\end{proposition}

In fact it turns out that $\Omega$ and $\log \tau$ have normal orders -- for more on normal orders see Tenenbaum~\cite{Tenenbaum}.

\section{Proof of Theorem \ref{main}(i)}\label{thm1}

We show that $P(G) \sim n$ for any graph $G$ with $V(G) = [n]$ and maximum degree 
\[ d \leq \sqrt{n}(\log n)^{-\log 2} 2^{-\omega(n)\sqrt{\log \log n}},\] 
where $\omega(n) \rightarrow \infty$ and $\omega(n) \leq \sqrt{\log\log n}$. This in turn shows that Theorem \ref{main}(i) holds for any $a > \log 2$.
Let $m = \lceil 4n/\omega(n) \rceil$, and let $L$ be the set of the first $n + m$ natural numbers with at most
\[ t = \omega(n)^{-1}(\log n)^{\log 2}2^{\omega(n)\sqrt{\log \log n}}\]
divisors. Since $\omega(n) \leq \sqrt{\log \log n}$ and $m \ll n$, Proposition \ref{tau} shows that 
\[ \max L \sim n + m \sim n.\]
Now uniformly and randomly select an $n$-element subset $\{\ell_1,\ell_2,\dots,\ell_n\}$ of $L$ and define the injective
labeling $\chi(i) = \ell_i$. For an unordered pair $\{xy,uv\} \in {E(G) \choose 2}$ of distinct edges of $G$, let $A_{xy,uv}$ be the event $\chi(x)\chi(y) = \chi(u)\chi(v)$ and define
\[ J_{xy,uv} = \{A_{jk, rs} :  \{j,k,r,s\} \cap \{x,y,u,v\} \not = \emptyset\} \quad \mbox{ and } \quad K_{xy,uv} = {E(G) \choose 2} \setminus J_{xy,uv} \cup \{A_{xy, uv}\}.\]
We apply the modified local lemma, Proposition \ref{mod}, to the events $A_{xy,uv}$. Fixing any set $K \subset K_{xy,uv}$, the set
\[ M := L \backslash \{\chi(z) : z \in V(G) \backslash \{u,v,x,y\}\}\]
has size at least $m + 4$. For any labels $\chi(u),\chi(v) \in M$, the number of ways of choosing $\chi(x),\chi(y) \in M$ such that
$\chi(x)\chi(y) = \chi(u)\chi(v)$ is at most
\[ \tau(\chi(u)\chi(v)) \leq \tau(\chi(u))\tau(\chi(v)) \leq t^2.\]
Therefore
\[ \mathbb P(A_{xy,uv} \; | \; \bigcap_{\{jk,rs\} \in K} \overline{A}_{jk,rs}) \leq \frac{t^2}{{|M| \choose 2}} < \frac{2t^2}{m^2}.\]
Note that regardless of the labeling of $u,v,x$ and $y$, none of the events $A_{jk,rs} : \{jk,rs\} \in K$ occur, since only edges
which contain at least one of $u,v,x$ and $y$ are affected by the labeling of $u,v,x$ and $y$. For any $\{xy,uv\} \in {E(G) \choose 2}$,
\[ |J_{xy,uv}| \leq 4d|E(G)| \leq 2d^2n.\]
Taking $\gamma = 1/(4d^2 n)$, and using $m^2 \geq 16d^2 t^2 n$, we find
\[ \gamma(1 - \gamma \max|J_{xy,uv}|) \geq \frac{1}{8d^2 n} \geq \frac{2t^2}{m^2}.\]
By the modified local lemma, the probability that none of the events $A_{xy,uv}$ occur is positive. In other words, there exists
a product-injective labeling $\chi : V(G) \rightarrow [N]$ where $N = \max L \sim n$. \qed

\section{Proof of Theorem \ref{main}(ii)}\label{thm2}

In this section, we prove that labeling with primes is asymptotically optimal for graphs that are much less dense than $K_n$, namely
for the random graph $G_{n,d/n}$ with $d \geq \sqrt{n}(\log n)^b$ and $b > 4.5$. Since $G_{n,d/n}$ for $d \geq \sqrt{n}(\log n)^b$ 
has maximum degree asymptotic to $d$, this is enough for Theorem \ref{main}(ii). Throughout this section, if $H$ is a graph then $C_4(H)$ is the number of 4-cycles in $H$.

\subsection{Counting 4-cycles}

\begin{lemma}\label{saturation}
Let $B = B(U,V)$ be a bipartite graph with $|U| = m$ and $|V| = n$, and let $d$ be the average degree of the vertices in $V$. If $nd^2 \geq 4m^2$ and $d \geq 2$, then
\[ C_4(B) \geq \frac{n^2 d^4}{64m^2}.\]
\end{lemma}

\begin{proof}
This is a standard exercise in applying Jensen's Inequality, but we include the proof for completeness. Let $M = {m \choose 2}$ and let $d(u,v)$ be the codegree of $u$ and $v$, that is, the number of vertices of $B$ adjacent to both $u$ and $v$. Then the number $C_4(B)$ of 4-cycles in $B$ is precisely
\[ C_4(B) = \sum_{\{u,v\} \subset U} {d(u,v) \choose 2}.\]
Let $M = {m \choose 2}$. By Jensen's Inequality, and since
\[ \sum_{\{u,v\} \subset U} d(u,v) = \sum_{w \in V} {d(w) \choose 2},\]
we have
\[ C_4(B) \geq M { \frac{1}{M} \sum_{w \in V} {d(w) \choose 2} \choose 2}.\]
By Jensen's Inequality again,
\[ \sum_{w \in V} {d(w) \choose 2} \geq n{d \choose 2}.\]
Therefore
\[ C_4(B) \geq M {\frac{n}{M}{d \choose 2} \choose 2}.\]
Since $nd^2 \geq 4m^2$, and ${x \choose 2} \geq \frac{1}{4}x^2$ for $x \geq 2$,
\[ \frac{n}{M} {d \choose 2} \geq \frac{nd^4}{2m^2} \geq 2.\]
Using ${x \choose 2} \geq \frac{1}{4}x^2$ again for $x \geq 2$,
\[ C_4(B) \geq \frac{n^2 d^4}{64M} \geq \frac{n^2 d^4}{64m^2}.\]
This proves the lemma.
\end{proof}

\subsection{Counting solutions to $uv = xy$}

A solution to $uv = xy$ is non-trivial if $\{u,v\} \neq \{x,y\}$.

\begin{lemma}\label{productsat}
For all $\varepsilon > 0$, there exist $\delta > 0$ and $n_0(\varepsilon)$ such that for $n \geq n_0(\varepsilon)$, if $A\subset [n]$ and $|A| \geq (1+\varepsilon) n/\log n$, then the number of non-trivial quadruples $\{a,b,c,d\}\in \binom{A}{4}$ satisfying $ab = cd$ is at least $\delta n^2 (\log n)^{-8}$.
\end{lemma}

\begin{proof} Let $n_0(\varepsilon)$ be the smallest positive integer such that for $n \geq n_0(\varepsilon)$,
\begin{eqnarray}
\frac{\varepsilon\sqrt{n}}{(2\log n)^2} &\geq& 2. \label{n1} \\
\frac{\varepsilon^2 n^2}{(2\log n)^4} &\geq& \frac{n^2}{(\log n)^6}. \label{n2} \\
(1 + \tfrac{1}{2}\varepsilon)\frac{n}{\log n} &\geq& \pi(n). \label{n3} \\
\pi(n) - \pi\Big(\frac{n}{2(\log n)^3}\Bigr) - \pi(n^{2/3}) &\geq& 4(\log n)^6. \label{n4}
\end{eqnarray}
Note that the last pair of inequalities is possible since $\pi(n) \sim n/\log n$ by the Prime Number Theorem. We shall prove the lemma with $\delta = 2^{-14}\varepsilon^4$ and $n \geq n_0(\varepsilon)$.

\medskip

Consider the bipartite graph $H = H(U,V)$ with parts $U = [n^{2/3}] \cup P$ and $V = [\sqrt{n}]$, where $P$ comprises the primes in the interval $[n^{2/3}, n]$, 
and where $uv \in E(H)$ if there exists $a \in A$ such that $a = uv$ with $u \geq v$. 
Erd\H{o}s~\cite{E} made the following observation:

\begin{center}
{\em for any $a\in [n]$, there exist $u \in U$ and $v \in V$ such that $a = uv$ and $u \geq v$.}
\end{center}

Consequently, $|E(H)| = |A|$. If $\{uv,vw,wx,xu\}$ is a 4-cycle in $H$, then $uv,wx \in A$ and $vw,xu \in A$ and 
$(uv)(wx) = (vw)(ux)$. Therefore $C_4(H)$ is the number of non-trivial solutions to $ab = cd$ with $a,b,c,d \in A$. 
For the remainder of the proof, we show $C_4(H) \geq \delta n^2 (\log n)^{-8}$. 

\medskip

Let $k_0$ be the largest integer $k$ such that $2^{k + 1} < \sqrt{n}/(\log n)^3$. For $1 \leq k \leq k_0$, let
\[ U_k:=\{u\in U:  2^{k-1} \sqrt{n} \leq u < 2^k \sqrt{n}\} \quad \mbox{ and } \quad  V_k:= \{v\in V: v \leq 2^{1-k}\sqrt{n}\}.\]
Denote by $H_k$ the subgraph of $H$ induced by $U_k$ and $V_k$. Also, let $H_0$ be the subgraph of $H$ induced by 
$U_0$ and $V_0$, where 
\[ U_0 = \{u \in U : u > n/2(\log n)^3\} \quad \mbox{ and } \quad V_0 = \{v \in V : v \leq 2(\log n)^3\}.\]
Then $H = \bigcup_{k = 0}^{k_0} H_k$. We consider the subgraphs $H_k : k \geq 1$ separately from $H_0$. 

\medskip

{\bf Claim 1.} {\em If for some $k \in [k_0]$, $|E(H_k)| \geq \varepsilon n/(2\log n)^2$, then $C_4(H_k) \geq \delta n^2(\log n)^{-8}$.} \\
{\it Proof.} The average degree in $H_k$ of vertices in $V_k$ is
\[ d = \frac{|E(H_k)|}{|V_k|} \geq \frac{\varepsilon 2^k \sqrt{n}}{(2\log n)^{2}}.\]
Since $n \geq n_0(\varepsilon)$, (\ref{n1}) gives $d \geq 2$ and (\ref{n2}) gives $|V_k| d^2 = |E(H_k)|^2 \geq 4|U_k|^2$.
By Lemma \ref{saturation} with $m = |U_k|$,
\[ C_4(H_k) \geq \frac{\varepsilon^4 n^2}{2^{14}(\log n)^{8}} = \delta n^2 (\log n)^{-8}.\]
This proves the claim. \qed

\medskip

Since $C_4(H) \geq C_4(H_k)$, we are done if $|E(H_k)| \geq \varepsilon n/(2\log n)^2$ for some $k \in [k_0]$, so we assume this is not the 
case for any $k \in [k_0]$. Then
\begin{equation}\label{sumhk}
\sum_{k = 1}^{k_0} |E(H_k)| \leq \frac{\varepsilon k_0 n}{(2\log n)^2} < \frac{\varepsilon n}{4\log n}.
\end{equation}
Next we consider $H_0$. 

\medskip

{\bf Claim 2.} {\em If $C_4(H) < \delta n^2(\log n)^{-8}$, then }
\begin{equation}\label{H0}
|H_0| < \pi(n) + \frac{16\delta^{1/2}n}{\log n}. 
\end{equation}
{\it Proof.} Let $\tilde{U}_0$ comprise all vertices of $U_0$ of degree at least two in $H_0$ and let $\tilde{H}_0$ be 
the subgraph of $H_0$ induced by $\tilde{U}_0 \cup V_0$. Then
\begin{equation}\label{degr1}
|E(H_0 \backslash \tilde{H}_0)| \leq |U_0| \leq \pi(n)
\end{equation}
Let $d$ be the average degree in $\tilde{H}_0$ of the vertices in $\tilde{U}_0$. Then $d \geq 2$ and by (\ref{n4}), 
$|\tilde{U}_0|d^2 \geq 4|\tilde{U}_0| \geq 4|V_0|^2$. By Lemma \ref{saturation} with $m = |V_0|$,
\begin{equation}\label{c4}
C_4(\tilde{H}_0) \geq \frac{|\tilde{U}_0|^2 d^4}{64m^2} \geq \frac{|E(\tilde{H}_0)|^2}{64m^2}.
\end{equation}
Since $C_4(H_0) < \delta n^2 (\log n)^{-8}$ and $m \leq 2(\log n)^3$, (\ref{c4}) gives
\[ |E(\tilde{H}_0)| < 8\delta^{1/2}(\log n)^{-4}mn < \frac{16\delta^{1/2}n}{\log n}.\]
Together with (\ref{degr1}), this completes the proof of Claim 2. \qed

\medskip

We now complete the proof of the lemma. By (\ref{sumhk}) and (\ref{H0}), 
\[ |E(H)| \leq \sum_{k = 0}^{k_0} |E(H_k)| < \pi(n) + \frac{16\delta^{1/2} n}{\log n} + \frac{\varepsilon n}{4\log n}.\]
Since $\delta = 2^{-14}\varepsilon^4$, the last two terms above are at most $\varepsilon n/2\log n$.
By (\ref{n3}), $\pi(n) \leq (1 + \varepsilon/2)n/\log n$, so we conclude $|E(H)| < (1 + \varepsilon)n/\log n$.
Since $|A| = |E(H)|$ and $|A| \geq (1 + \varepsilon)n/\log n$, this contradiction completes the proof.
\end{proof}

\subsection{Proof of Theorem \ref{main}(ii)}

Let $\varepsilon > 0$, and let $\delta$ be given by Lemma \ref{productsat}. Let $p = 2(\log n)^{4.5}/(\delta\sqrt{n})$.
For a fixed labeling $\chi : V(G_{n,p}) \rightarrow \mathbb Z$, let $\mathbb P(\chi)$ denote the probability that $\chi$ is a product-injective labeling of $G = G_{n,p}$. To prove Theorem 
\ref{main}(ii), we show that if $n \geq (1 + \varepsilon)N/\log N$, then the expected number $E$ of product-injective labelings $\chi : V(G) \rightarrow [N]$ satisfies
\begin{equation}\label{expected}
 E = \sum_{\chi : V(G) \rightarrow [N]} \mathbb P(\chi) \leq {N \choose n} \max_{\chi} \mathbb P(\chi) \ll 1.
 \end{equation}
To prove this, we show $\mathbb P(\chi) \ll N^{-n}$ for every fixed labeling $\chi : V(G) \rightarrow [N]$. Let $k\in [N^2]$, and let $g_k$ be the number of representations (for the given function $\chi$) of the form $k = \chi(i)\chi(j)$.  Then for $\chi$ to be product-injective, for each $k$, at most one of the $g_k$ possible edges
 $\{i,j\}$ with $k = \chi(i)\chi(j)$ may be selected to be in the random graph $G$.  Therefore,
\begin{equation}\label{probability product}
\mathbb P(\chi)= \prod_{k=1}^{N^2} (1-p)^{g_k} + g_k p(1-p)^{g_k-1}.
\end{equation}
For a real-valued function $f$, let $f_+ = \max\{f,0\}$.  Then by Lemma \ref{productsat}, if $n \geq n_0(\varepsilon)$, then 
\begin{equation}\label{sum of g_k}
\sum_{k=1}^{N^2} (g_k - 1)_+ \geq \frac{\delta n^2}{(\log n)^8}.
\end{equation}
If $g_i \geq g_j+2$, then \eqref{probability product} increases by replacing $g_i$ with $g_i -1$ and $g_j$ with $g_j+1$. So by \eqref{sum of g_k}
\[
\mathbb P(\chi) \leq \left((1-p)^2 + 2p(1-p)\right)^g = (1-p^2)^{g/2} \leq e^{-p^2 g/2},
\]
where $g = \delta n^2(\log n)^{-8}$.  Since $p^2 g \geq 4n\log n$, $\mathbb P(\chi) \leq n^{-2n} \ll N^{-n}$. 
This proves (\ref{expected}), and completes the proof of Theorem \ref{main}(ii). \qed


\begin{thebibliography}{99}

\bibitem{AS} N. Alon, J. Spencer. The Probabilistic Method, Wiley, New York (2000).

\bibitem{B} B. Bollob\' as. Random Graphs, Cambridge university press, 2001.

\bibitem{BP} B. Bollob\' as and O. Pikhurko. Integer sets with prescribed pairwise differences being distinct.
{\em European J. Combin.} 26 (2005), no. 5, 607 -- 616.

\bibitem{E} P. Erd\H os, On sequences of integers no one of which divides the product of two others and on some related problems. {\em Tomsk. Gos. Univ. Ucen. Zap.} {\bf 2} (1938), 74-82.


\bibitem{HR} G. Hardy and S. Ramanujan. The normal number of prime factors of a number $n$. {\em Quarterly Journal of Mathematics} 48 (1917), 76--92.



\bibitem{TV1} M. Tait, J. Verstra\"{e}te. Sum-injective labellings of graphs. Preprint (2014).

\bibitem{Tenenbaum}  G. Tenenbaum. Introduction to analytic and probabilistic number theory, Cambridge, Cambridge University Press, 1995, Cambridge studies in advanced mathematics, No. 46.


\end{thebibliography}
\end{document}